\documentclass[oneside,11pt]{amsart}
\usepackage{amssymb,amscd,amsmath,latexsym}
\usepackage[mathcal]{euscript}
\usepackage{url}

\newtheorem{thm}{Theorem}[subsection]
\newtheorem{cor}[thm]{Corollary}

\newtheorem{prop}[thm]{Proposition}
\theoremstyle{definition}

\theoremstyle{remark}

\newcommand{\Ext}{\operatorname{Ext}}

\newcommand{\opH}{\operatorname{H}}

\newcommand{\gfp}{G(\mathbb{F}_p)}
\newcommand{\gfpr}{G(\mathbb{F}_{q})}

\newcommand{\cgr}{\mathcal{G}_r(k)}

 \numberwithin{equation}{subsection}
 \numberwithin{table}{subsection}

\begin{document}

\title[Cohomology for algebraic groups, finite groups and Lie algebras]
{\bf Cohomology of algebraic groups, finite groups, and Lie algebras: Interactions and Connections}

\author{\sc Daniel K. Nakano}
\address
{Department of Mathematics\\ University of Georgia \\
Athens\\ GA~30602, USA}
\thanks{Research of the author was supported in part by NSF
grant  DMS-0654169} \email{nakano@math.uga.edu}

\date{}

\maketitle

\section{Overview} 

\subsection{} Let $G$ be a simple, simply connected algebraic 
group scheme defined over finite field ${\mathbb F}_{p}$, and 
$F^{r}:G\rightarrow G$ be the $r$th iteration of the Frobenius morphism. 
Let $G_{r}$ be the scheme theoretic kernel of $F^{r}$, and $G({\mathbb F}_{q})$ 
be the finite Chevalley group obtained by looking at the fixed points under 
$F^{r}$. For a given rational $G$-module $M$ one can consider 
the restriction of the action of $M$ to the infinitesimal Frobenius kernel 
$G_r$ and to the finite group $G({\mathbb F}_q)$.  

\begin{figure}[ht]\label{proj1} 
\setlength{\unitlength}{.5cm}
\begin{center}
\begin{picture}(13,6)
\put(5.1,5.2){$\text{mod}(G)$}
\put(1,1.2){$\text{mod}(G_{r})$} 
\put(9.5,1.2){$\text{mod}(G({\mathbb F}_{q}))$} 
\put(5.4,4.7){\vector(-1,-1){2.5}}
\put(7.6,4.7){\vector(1,-1){2.5}}
\put(2.5,3.6){${\bf res}$}
\put(9.5,3.6){$\bf res$}
\end{picture}
\caption{}
\end{center}
\end{figure}

For over 40 years, representation theorists have attempted to understand the relationship between 
these three module categories. Two basic questions we will address in this paper are:  
\vskip .25cm 
\noindent 
(1.1.1) What is the relationship
between the representation theories of these categories?
\vskip .25cm 
\noindent 
(1.1.2) Is there a relationship between the cohomology theories of 
these categories?
\vskip .25cm 
Even though there is no direct functorial relationship 
between $\text{mod}(G_{r})$ and $\text{mod}(G({\mathbb F}_{q}))$ these categories have 
very similar features. For example, both categories correspond to module categories for finite dimensional 
cocommutative Hopf algebras. Curtis proved in the 1960s there exists 
a one-to-one correspondence between the simple $G_{r}$-modules and 
simple $G({\mathbb F}_{q})$-modules. This correspondence is given by simply restricting 
natural classes of simple $G$-modules. Furthermore, Steinberg's Tensor Product Theorem allows one to easily transfer the 
questions pertaining to the computation of irreducible characters between these three categories.
The theme of lifting $G_{r}$-modules to $G$-structures was further studied by Ballard and 
Jantzen for projective $G_{r}$-modules. For a general guide to the history and results 
in this area we refer the reader to \cite{Hu2}.

The aim of this paper is to present new methods in relating these three categories which 
have been developed over the past 10 years. Even though the group $G({\mathbb F}_{p})$ is finite, 
we will see that there are ambient geometric stuctures which govern its representation theory 
and cohomology. One theme in these new approaches has been to 
use nilpotent orbit theory for semisimple Lie algebras and the geometry of the flag 
variety $G/B$ to calculate extensions of modules over the group $G({\mathbb F}_{p})$. 
Often times the use of such tools can reduce difficult group cohomology calculations to combinatorics 
within the given root system.

\subsection{} This paper is organized as follows. In Section 2, we review the basic facts about the representation theory of $G$, $G_{r}$ and $G({\mathbb F}_{q})$. We give a parametrization of the simple modules for each of these cases and describe how the simple modules are related. The projective modules for $G_{r}$ and 
$G({\mathbb F}_{q})$ are discussed in relation to proving the existence of $G$-structures. 
In the following section (Section 3), we give the fundamental definitions and properties of the 
complexity and support varieties of a module for a finite dimensional cocommutative Hopf algebra 
(i.e., finite group scheme). In Section 4, we discuss the relationships between the complexity 
and support varieties of a given finite dimensional $G$-module when restricted 
to $\text{mod}(G_{1})$ and $\text{mod}(G({\mathbb F}_{p}))$. These results were motivated 
by questions posed by Parshall in the 1980s. In Section 5, we investigate the connections between 
the cohomology groups in $\text{mod}(G)$, $\text{mod}(G_{r})$ and $\text{mod}(G({\mathbb F}_{q}))$ using 
techniques developed by Bendel, Pillen and the author. Formulas for $\text{Ext}^{1}$-groups will be exhibited, 
in addition to results pertaining to the existence of self-extensions. 
Finally, in Section 6, we present new developments in the computation of the cohomology groups 
$\text{H}^{\bullet}(G({\mathbb F}_{q}),k)$. In the case when 
$r=1$ and $p$ is larger than the Coxeter number, we give an upper bound 
for the dimension of these groups via Kostant's Partition Functions. Moreover, 
our methods allow us to give precise information about the existence of the first non-trivial cohomology 
when the root system is of type $A$ or $C$. 

The author would like to acknowledge his many discussions with and 
the important contributions to the subject in recent years by 
Henning Andersen, Christopher Bendel, Jon Carlson, Eric Friedlander, Zongzhu Lin, 
Brian Parshall, Cornelius Pillen, and Leonard Scott. The author would also like to thank Bin Shu and his staff for their efforts in hosting and organizing a vibrant summer school program at East China Normal University in Summer 2009 when the results in this paper were presented.

\section{Representation Theory} 

\subsection{Notation: } Throughout this paper, we will follow the basic Lie theoretic conventions 
used in \cite{Jan1}. Let $G$ be a simple simply connected algebraic group scheme which is defined
and split over the finite field ${\mathbb F}_p$ with $p$ elements. Moreover, let $k$ denote an algebraically 
closed field of characteristic $p$. Let $T$ be a maximal split torus and $\Phi$ be the root system
associated to $(G,T)$. The positive (resp. negative)
roots will be denoted by  $\Phi^{+}$ (resp. $\Phi^{-}$), with $\Delta$ the collection of simple roots. 
Let $B$ be a Borel subgroup containing $T$ corresponding to the negative roots and $U$ be the
unipotent radical of $B$. For a given root system of rank $n$,
the simple roots in $\Delta$ will be denoted by $\alpha_1, \alpha_2, \dots, \alpha_n$.
Our convention will be to employ the Bourbaki ordering of simple roots. In 
particular, for type $B_n$, $\alpha_n$ denotes the unique short simple root and for type $C_n$, $\alpha_n$ denotes the 
unique long simple root. The highest short root is labelled by $\alpha_{0}$. 

Let ${\mathbb E}$ be the Euclidean space associated with $\Phi$, and
the inner product on ${\mathbb E} $ will be denoted by $\langle\ , \
\rangle$. Set $\alpha^{\vee}=2\alpha/\langle\alpha,\alpha\rangle$ (the coroot corresponding to $\alpha\in \Phi$).
The fundamental weights (basis dual to $\alpha_1^{\vee}, \alpha_2^{\vee}, \dots, \alpha_n^{\vee}$)
will be denoted by $\omega_1$, $\omega_2$, \dots,
$\omega_n$. Let $W$ denote the Weyl group associated to $\Phi$ and for $w \in W$, let $\ell(w)$
denote the length of the word.

Let $X(T)$ be the integral weight lattice spanned by
the fundamental weights. For $\lambda,\mu\in X(T)$, $\lambda\geq \mu$ if and only if $\lambda-\mu$ is a positive 
integral linear combination of simple roots. The set of dominant integral weights is defined by 
$$X(T)_{+}:=\{\lambda\in X(T):\  0\leq\langle \lambda,\alpha^{\vee}\rangle , \ \text{for all $\alpha\in\Delta$} \}.$$ 
For a weight $\lambda \in X(T)$, set $\lambda^* := -w_0\lambda$ 
where $w_0$ is the longest word in the 
Weyl group $W$.  The dot action is defined by $w \cdot \lambda := w(\lambda + \rho) -\rho$ on $X(T)$ where  
$\rho$ is the half-sum of the positive roots. For $\alpha \in \Delta$, $s_{\alpha} \in W$ denotes the 
reflection in the hyperplane determined by $\alpha$. The Coxeter number $h$ equals $\langle 
\rho,\alpha_0^{\vee} \rangle +1$. 

Let $M$ be a $T$-module. Then $M=\oplus M_{\lambda}$ where $M_{\lambda}$ are the weight spaces of $M$. The 
formal character of $M$ is $\text{ch }M=\sum_{\lambda\in X(T)} \dim M_{\lambda}\ e(\lambda)$. An important 
class of $G$-modules whose formal character is known are the induced or Weyl modules. For $\lambda\in X(T)_{+}$, 
let 
$$H^0(\lambda) := \text{ind}_B^G\lambda=\{f:G\rightarrow\lambda:\ f(g\cdot b)=\lambda(b^{-1})f(g),\ \text{
for all $b\in B$, $g\in G$}\}$$ 
be the induced module. These $G$-modules can be interpreted geometrically as the global sections of 
the line bundle ${\mathcal L}(\lambda)$ over the flag variety $G/B$. The Weyl module of highest weight 
$\lambda$ is defined by $V(\lambda) := H^0(\lambda^*)^*$. The following theorem provides basic properties of 
the induced/Weyl modules. 

\begin{thm}
\begin{itemize}
\item[(a)] $H^0(\lambda)\neq 0$ if and only if $\lambda\in X(T)_+$;
\item[(b)] $\dim H^0(\lambda)_{\lambda}=1$, $H^0(\lambda)_{\mu} \neq 0$ implies 
that $w_{0}\lambda\leq \mu\leq\lambda$;
\item[(c)] $\dim H^0(\lambda)=\prod\limits_{\alpha\in\Phi^+}
\frac{\langle \lambda+\rho,\alpha^{\vee}\rangle}{\langle \rho,\alpha^{\vee}\rangle}$. 
\end{itemize}
\end{thm}

If $H$ is an algebraic group scheme defined over ${\mathbb F}_{p}$ then there exists a Frobenius 
morphism $F:H\rightarrow H$. In the case when $H=GL_{n}(-)$, $F$ takes a matrix $(a_{ij})$ to 
$(a_{ij}^{p})$. Let $F^{r}$ be the $r$-th iteration of the $F$ with itself. For a rational 
$H$-module $M$, let $M^{(r)}$ be the module obtained by composing the underlying representation for $M$ 
with $F^{r}$. Moreover, let $M^*$ denote the dual module. 

We will be interested in understanding the representation theory and cohomology of infinitesimal 
Frobenius kernels for reductive algebraic groups $G$. For $r\geq 1$, let $G_r:=\text{ker }F^{r}$ be the 
$r$th Frobenius kernel of $G$ and $\gfpr$ be the associated finite Chevalley group obtained by looking 
at the ${\mathbb F}_{q}$ rational points of $G$. We note that in the case when $r=1$, $\text{Mod}(G_{1})$ 
is equivalent to $\text{Mod}(u({\mathfrak g}))$ where ${\mathfrak g}=\text{Lie }G$ and $u({\mathfrak g})$ is 
the restricted enveloping algebra of ${\mathfrak g}$. 

\subsection{Simple Modules: } In this section we will describe the classification of simple 
modules in $\text{mod}(G)$, $\text{mod}(G_{r})$ and $\text{mod}(G({\mathbb F}_{q}))$. All 
modules will be defined over the field $k$. The details can be found in \cite[II. Chapter 2-3]{Jan1}. 

First for $\lambda\in X(T)_{+}$,  the socle of $H^{0}(\lambda)$ (largest semisimple submodule) 
is simple. Set $L(\lambda)=\text{soc}_G\ H^0(\lambda)$. Then the simple finite-dimensional 
modules are given by: $\{L(\lambda):\ \lambda\in X(T)_{+}\}$. Note that when 
$k=\mathbb{C}$ then $L(\lambda)\cong H^0(\lambda)$ for all $\lambda\in X(T)_+$.

In order to classify the simple modules in $\text{mod}(G_{r})$ and $\text{mod}(G({\mathbb F}_{q}))$ we need to 
define the set of $p^{r}$-restricted weights. Let 
$$X_r(T)=\{\lambda\in X(T)_+:\ 0\leq\langle \lambda,\alpha^{\vee} \rangle \leq p^r-1,\ \text{for all $\alpha\in\Delta$}\}.$$
If $\lambda\in X_{r}(T)$ then $L(\lambda)|_{G_{r}}$ remain simple, and 
\vskip .25cm 
\noindent 
(2.2.1) $\{\text{simple finite-dimensional modules for $G_r$}\}\leftrightarrow \{L(\lambda)\downarrow_{G_r}:\ \lambda\in X_r(T)\}$ (one to one correspondence).
\vskip .15cm 
\noindent  
On the other hand, Curtis proved that one can restrict $L(\lambda)$ to $G({\mathbb F}_{q})$ and also get 
a classification of simple $G({\mathbb F}_{q})$-modules: 
\vskip .25cm 
\noindent 
(2.2.2) $\{\text{simple finite-dimensional modules for $G({\mathbb F}_{q})$}\}\leftrightarrow 
\{L(\lambda)\downarrow_{G({\mathbb F}_{q})}:\ \lambda\in X_r(T)\}$ (one to one correspondence).

Now by a beautiful theorem of Steinberg (known as Steinberg's Tensor Product Theorem), one can describe 
all simple $G$-modules by tensoring by Frobenius twists of simple $G_{1}$-modules: 

\begin{thm} Let $\lambda\in X(T)_+$, $\lambda=\lambda_0+\lambda_1p+\cdots+\lambda_sp^s$,
where $\lambda_i\in X_1(T)$. Then 
$$L(\lambda)\cong L(\lambda_0)\otimes L(\lambda_1)^{(1)}\otimes\cdots
\otimes L(\lambda_s)^{(s)}.$$
\end{thm}

If $G$ is semisimple then $X_1(T)$ is finite. As a consequence of the Steinberg Tensor Product Theorem, 
if we know the characters of the simple $G_{1}$-modules $L(\lambda)$ where $\lambda\in X_1(T)$, we can 
compute the characters of all simple finite-dimensional $G$-modules. The characters of the 
simple $G-$modules are still undetermined. For $p>h$ there is a conjecture due to Lusztig which give 
a recursive formula for the characters of simple modules via Kazhdan-Lusztig polynomials (cf. \cite[Chapter C]{Jan1}).

\subsection{Projective Modules: } 

For $r\geq 1$, the infinitesimal Frobenius kernel $G_{r}$ is a finite group scheme whose 
module category is equivalent to $\text{Mod}(\text{Dist}(G_{r}))$ where $\text{Dist}(G_{r})$ 
is a finite-dimensional cocommutative Hopf algebra. On the other hand, the category of $G({\mathbb F}_{q})$ 
modules that we are considering is equivalent to $\text{Mod}(kG({\mathbb F}_{q}))$ where 
$kG({\mathbb F}_{q})$ is the group algebra of $G({\mathbb F}_{q})$ (also a finite-dimensional 
cocommutative Hopf algebra). In both instances the distribution algebra $\text{Dist}(G_{r})$ 
and $kG({\mathbb F}_{q})$ are symmetric algebra. This means that these algebras are self-injective 
and the projective cover of a simple module is isomorphic to its injective hull. 

If $L(\lambda)$ ($\lambda\in X_r(T)$) is a simple module for $G_r$ then let
$Q_r(\lambda)$ be its injective hull (and projective cover). One natural question to 
ask is whether we can lift the $G_{r}$-structure on $Q_{r}(\lambda)$ to a compatible 
$G$-module structure. Ballard showed that this holds when $p\geq 3(h-1)$, and Jantzen 
later lowered the bound to $p\geq 2(h-1)$. 

\begin{thm} Let $\lambda\in X_{r}(T)$ and $p\geq 2(h-1)$. Then $Q_{r}(\lambda)$ admits a $G$-structure. 
\end{thm} 

The idea behind the proof given in \cite[II 11.11]{Jan1} is to first realize $Q_{r}(\lambda)$ as a 
$G_r$-direct summand of the $G$-module $\text{St}_r\otimes L$, where $L$ is an irreducible $G$-module and $\text{St}_r$ is the $r$th 
Steinberg module. The proof entails showing that $Q_{r}(\lambda)$ is in fact a $G$-direct summand when $p\geq 2(h-1)$. 

The Steinberg module $\text{St}_{r}$ is projective upon restriction to $G_{r}$ and $G({\mathbb F}_{q})$. Therefore, 
$\text{St}_{r}\otimes L$ is projective upon restriction to $G({\mathbb F}_{q})$. Consequently, for $p\geq 2(h-1)$, 
one can lift $Q_{r}(\lambda)$ to a $G$-structure and upon restriction to $G({\mathbb F}_{q})$ this module will 
split into a direct summand of projective indecomposable $kG({\mathbb F}_{q})$-modules. For $\lambda\in X_{r}(T)$, 
let $U_{r}(\lambda)$ be the projective cover of $L(\lambda)$ in $\text{mod}(kG({\mathbb F}_{q}))$. 
Chastkofsky \cite{Ch} proved that 
$$[Q_{r}(\lambda)|_{G({\mathbb F}_{q})}: U_{r}(\mu)] = \sum_{\nu \in \Gamma} [L(\mu) \otimes L(\nu) 
:L(\lambda)\otimes L(\nu)^{(r)}]_{G}.$$ 
Here $\Gamma=\{\nu\in X(T)_{+}:\ \langle \nu,\alpha_{0}^{\vee} \rangle <h\}$. We shall see later 
that similar factors appear in $\text{Ext}^{1}$-formulas relating the extensions of simple modules 
for $G({\mathbb F}_{q})$ to simple modules for $G$.

\section{Homological Algebra} 

\subsection{Complexity: } Let ${\mathcal V}=\{V_{j}:\ j\in {\mathbb N} \}$ be a sequence 
of finite-dimensional $k$-vector spaces. The rate of growth of 
${\mathcal V}$, denoted by $r({\mathcal V})$, is the smallest positive 
integer $c$ such that $\dim_{k} V_{n}\leq K\cdot n^{c-1}$ for some 
constant $K>0$. 

Given a finite-dimensional algebra $A$, Alperin \cite{Al} introduced the notion of the complexity of a module 
in 1977 as a generalization of the notion of periodicity in group cohomology. The complexity is 
defined as follows. Let $M\in \text{mod}(A)$ and 
\begin{equation*} 
\dots \rightarrow P_2 \rightarrow P_1 \rightarrow P_0 \rightarrow M \rightarrow 0
\end{equation*} 
be a minimal projective resolution of $M$. This means that the exact sequence above is a projective 
resolution of $M$ and the kernels have no projective summands. The complexity 
of $M$ in  $\text{mod}(A)$ is $r(\{{P}_{t}:\ t\in {\mathbb N}\})$. 

Let  $\text{Ext}^{j}_{A}(M,-)$ 
be the $j$th right derived functor of $\text{Hom}_{A}(M,-)$ for 
$M$ in $\text{Mod}(A)$. If $A$ is an augmented algebra (i.e., $A$-admits the trivial module $k$), and 
$N$ is in $\text{Mod}(A)$ set 
$\text{H}^{j}(A,N)=\text{Ext}^{j}_{A}(k,N)$. One can relate the complexity of $N$ in terms of the 
rate of growth of extension groups. 

\begin{thm} Let $A$ be a finite-dimensional algebra and  
$S_{1},S_{2},\ldots S_{u}$ be a complete set of simple objects 
in $\operatorname{mod}(A)$ then 
$$c_{A}(M)=r(\{\operatorname{Ext}^{t}_{A}(\oplus_{j=1}^{u}S_{j},N):\ t\in {\mathbb N}\}).$$ 
\end{thm} 

The complexity becomes a finer invariant when one considers the case when 
$A$ is self-injective, in particular when $A$ is a finite-dimensional Hopf algebra. 
In this setting a module being injective is equivalent to it being projective, and 
the complexity is an invariant of projectivity.  

\begin{prop} Let $A$ be a self injective algebra. Then $c_{A}(M)=0$ if and only if $M$ is projective. 
\end{prop} 
\begin{proof}  
This can be shown as follows. Suppose $c_A(M)=0$, then
$$0\rightarrow P_s\rightarrow\cdots\rightarrow P_1\rightarrow P_0\rightarrow M\rightarrow0$$
is a minimal projective resolution of $M$. Furthermore, $P_s$ is injective so this sequence splits, 
hence $M$ is projective. On the other hand if $M$ is projective one can construct a resolution of 
$M$, $0\rightarrow M \rightarrow M \rightarrow 0$, thus $c_{A}(M)=0$. 
\end{proof} 

\subsection{Support Varieties: } Support varieties were introduced in the pioneering work of 
Carlson \cite{Ca1} \cite{Ca2} nearly 30 years ago as a method to study complexes and resolutions 
of modules over group algebras. Since that time the theory of support varieties 
has been extended to restricted Lie algebras by Friedlander and Parshall \cite{FP1}, to 
finite-dimensional sub Hopf algebras of the Steenrod algebra by Palmieri and the 
author \cite{NPal}, 
to infinitesimal group schemes by Suslin, Friedlander and Bendel 
\cite{SFB1} \cite{SFB2}, and to arbitrary finite-dimensional cocommutative Hopf algebras by 
Friedlander and Pevtsova \cite{FPe}. Further attempts to generalize 
the theory have been made to finite-dimensional algebras by Solberg 
and Snashall \cite{SS} via Hochschild cohomology and Hecke algebras by Erdmann 
and Holloway \cite{EH}, Lie superalgebras \cite{BKN1, BKN2, BKN3, BaKN}, 
quantum groups \cite{Ost} \cite{BNPP}, and quantum complete intersections \cite{BE}. As a testament to the dissemination of this theory, 
Dmitry Rumynin wrote in a Mathematical Review 
(MR: 2003b:20063) that ``support varieties have proved to be an indispensable 
tool in the arsenal of a modern representation theorist.'' 
  
Friedlander and Suslin \cite{FS} proved that for a finite-dimensional cocommutative 
Hopf algebra $A$, the even degree cohomology ring $R=\text{H}^{2\bullet}(A,k)$ is a commutative 
finitely generated algebra, and $\text{H}^{\bullet}(A,M)$ is a finitely generated $R$-module 
for any $M$ in $\text{mod}(A)$. Set ${\mathcal V}_{A}(k)=\text{Maxspec}(R)$. 
For any finitely generated $A$-module $M$, one can assign a conical subvariety 
${\mathcal V}_{A}(M)$ inside the spectrum of the cohomology ring ${\mathcal V}_{A}(k)$ 
by letting 
$${\mathcal V}_A(M)=\text{Maxspec}(R/J_M)$$
where $J_M=\text{Ann}_R \text{Ext}_A^\bullet(k,M\otimes M^{*})$. Here $M^{*}$ is the dual 
module of $M$.  

Support varieties provide a method to introduce the geometry of the 
spectrum of $R$ into the representation theory of $A$. These varieties have natural 
geometric properties and provide a geometric interpretation of the complexity of $M$: 
\vskip .25cm 
\noindent 
(3.2.1) $c_A(M)=\dim {\mathcal V}_A(M)$
\vskip .15cm 
\noindent 
(3.2.2) ${\mathcal V}_A(M\oplus N)={\mathcal V}_A(M)\cup {\mathcal V}_A(N)$ 
\vskip .15cm 
\noindent 
(3.3.3) ${\mathcal V}_A(M\otimes N)={\mathcal V}_A(M)\cap {\mathcal V}_A(N)$ 
\vskip .15cm 
\noindent
(3.3.4) If $W$ is a conical subvariety of ${\mathcal V}_{A}(k)$ there exists 
$M$ in $\text{mod}(A)$ such that ${\mathcal V}_{A}(M)=W$. 
\vskip .15cm 
\noindent 
(3.3.5) If $M\in \text{mod}(A)$ is indecomposable then $\text{Proj}({\mathcal V}_{A}(M))$ is 
connected. 
\vskip .25cm 
For finite groups a description of the spectrum of the cohomology ring was determined by 
Quillen \cite{Q1,Q2} using elementary abelian subgroups. Avrunin and Scott \cite{AS} proved a more 
general stratification theory for ${\mathcal V}_{A}(M)$. Support varieties 
become very transparent when one considers the case when 
${\mathfrak g}$ is a restricted Lie algebra over a field $k$ of characterstic $p>0$ 
and $A$ is the restricted universal enveloping algebra $u({\mathfrak g})$. 
In this situation the spectrum of the cohomology ring is homeomorphic 
to the restricted nullcone  
$${\mathcal N}_{1}({\mathfrak g})=\{x\in {\mathfrak g}:\ x^{[p]}=0\}.$$  
Moreover, Friedlander and Parshall showed that 
$${\mathcal V}_{G_1}(M)\cong\{x\in {\mathfrak g}:\ x^{[p]}=0,\ M|_{u(\langle 
x\rangle)}\ \text{is not free}\}\cup\{0\}.$$

If $G$ is a reductive algebraic group over $k$ and ${\mathfrak g}=\text{Lie }G$ 
then ${\mathcal N}_{1}({\mathfrak g})$ is a $G$-stable conical subvariety 
inside the cone of nilpotent elements (nullcone) ${\mathcal N}:={\mathcal N}({\mathfrak g})$. 
In the special case when $G=GL_{n}(k)$, 
$${\mathcal N}_{1}({\mathfrak g})=\{A\in \mathfrak{gl}_{n}(k):\ A^{p}=0\}.$$ 
The group $GL_{n}(k)$ acts on ${\mathcal N}$ via conjugation and 
has finitely many orbits (indexed by partitions of $n$). Note that 
${\mathcal N}_{1}({\mathfrak g})$ is $G$-stable. For general reductive groups $G$, the 
nullcone ${\mathcal N}$ has been well studied (see \cite{Car} \cite{CM} 
\cite{Hu2}) because of its beautiful geometric properties with deep connections to 
representation theory. The group $G$ acts on ${\mathcal N}$ via the adjoint representation 
and ${\mathcal N}$ has finitely many $G$-orbit (which are classified). In this setting support varieties behave well with respect to the $G$-action. 
\vskip .25cm 
\noindent 
(3.3.6) If $M$ is a $G$-module, then ${\mathcal V}_A(M)$ is $G$-invariant.
\vskip .15cm 
\noindent 
(3.3.7) If $M\in \text{mod}(G)$, ${\mathcal V}_A(M)=\overline{G\cdot
x_1}\cup\overline{G\cdot x_2}\cup\dots\cup\overline{G\cdot x_s}$. There are only 
finitely many possibilities for ${\mathcal V}_{A}(M)$ in this case because there are 
finitely many nilpotent $G$-orbits.

\section{Relating Support Varieties} 

\subsection{Motivation: } When I was a graduate student at Yale, my Ph.D advisor 
George Seligman told me a story about the Ph.D thesis of his former student 
David Pollack (circa 1966). Seligman said that he proposed the following problem to Pollack for 
his thesis: to determine the representation type for the restricted enveloping 
algebras $u({\mathfrak g})$ where ${\mathfrak g}=\text{Lie }G$ and $G$ is a 
reductive algebraic group. More specifically, at the time the question was to 
determine whether these algebras have finite or infinite representation type. 

For group algebras this problem was settled earlier: a group algebra over a field 
of characteristic $p>0$ has finite representation type if and only if its $p$-Sylow 
subgroups are cyclic. Therefore, when $G=SL_2(k)$, and $A=kG(\mathbb{F}_p)$, all 
Sylow subgroups are cyclic and conjugate to the subgroup of unipotent 
upper triangular matrices. Consequently, $A=kSL_{2}({\mathbb F}_{p})$ has finite representation type.

One of the first challenges was for Pollack to determine the 
representation type of $u(\mathfrak{sl}_{2})$. At the time Walter Feit 
thought by analogy that this algebra should be of finite representation type. 
But, after some efforts in trying to prove this, Pollack constructed the 
projective indecomposables 
for the algebra, and using these concrete realizations he constructed infinitely 
many indecomposable modules for $u(\mathfrak{sl}_{2})$. Moreover, Pollack was 
able to show that $u({\mathfrak g})$ where ${\mathfrak g}=\text{Lie }G$ 
($G$ is a reductive algebraic group) has infinite representation type. 

After hearing this story, I thought that there should be direct connections between 
understanding representation type between finite Chevalley groups and their 
associated Lie algebras. The pursuit in finding such connections was 
in part a motivating factor in the development of the results outlined in this section.

\subsection{Parshall Conjecture: } 

In 1987, Parshall (another former student of Seligman) asked the following 
question in his article in 
the Arcata Conference Proceedings. Let $G$ be a
reductive algebraic group and $M$ be in $\text{mod}(G)$. If $M$ is 
projective as a $G_{1}$-module then is $M$ projective over $G({\mathbb F}_{p})$?  
The affirmative version to this question has become known as the 
``Parshall Conjecture''. As we have seen in Section 2.3 the Parshall Conjecture 
holds if $p\geq 2(h-1)$. 

For this section, we will need some additional notion. Let 
$B$ be a Borel subgroup of $G$, $U$ be the unipotent radical of $B$. 
Furthermore, let $G(\mathbb{F}_p)$, $B(\mathbb{F}_p)$, and $U(\mathbb{F}_p)$ be the corresponding 
finite groups, and ${\mathfrak g}=\text{Lie }G$, ${\mathfrak b}=\text{Lie }B$, and 
${\mathfrak n}=\text{Lie }U$ be the corresponding Lie algebras.

We will now describe a process which allows us to pass from $\text{mod}(G_{1})$ to 
$\text{mod}(G({\mathbb F}_{p})$ without lifting to $\text{mod}(G)$. First consider 
$$U(\mathbb{F}_p)=\Gamma_1\supset\Gamma_2\supset\cdots\supset\Gamma_n=\{1\}$$
the lower central series where $\Gamma_i=[\Gamma_{i-1},U(\mathbb{F}_p)]$. 
Then the associated graded object 
$\text{gr }U(\mathbb{F}_p)=\bigoplus\limits_{n\geq1}\Gamma_n/\Gamma_{n+1}$
is a restricted Lie algebra (for most primes). We can 
identify the associated graded Lie algebra with ${\mathfrak n}$. That is, 
$\text{gr }U(\mathbb{F}_p)\cong {\mathfrak n}$.

Now consider filtering the group algebra of $U({\mathbb F}_{p})$ with powers 
of its augmentation ideal: 
$$
kU(\mathbb{F}_p)=I_0\supseteq I_1 \supseteq\cdots\supseteq I_t=\{0\}
$$
where $I_1=\text{rad }kU(\mathbb{F}_p)$. One can now form 
the associated graded algebra 
$\text{gr }kU(\mathbb{F}_p)=\bigoplus\limits_{n\geq0}I_n/I_{n+1}$. 
By using results of Quillen and Jennings, we have the following 
isomorphism of algebras: 
$$
\text{gr }kU(\mathbb{F}_p)\cong 
u(\text{gr }U(\mathbb{F}_p))\otimes_{\mathbb{F}_p}k\cong 
u({\mathfrak n}).
$$

The filtration on the group algebra of $U({\mathbb F}_{p})$ can be used to 
construct the May spectral sequence. For $M$ in $\text{mod}(U(\mathbb{F}_p))$,
$$E_1^{i,j}=\text{H}^{i+j}(u({\mathfrak n}),\text{gr }M)\Rightarrow
\text{H}^{i+j}(U(\mathbb{F}_p),M).$$
Here we are using the identification: $\text{gr }kU(\mathbb{F}_p)=u({\mathfrak n})$.
If $M$ in $\text{mod}(B)$,then $\text{gr }M=M|_{U_1}$ and we can rewrite the spectral sequence as 
$$E_1^{i,j}=\text{H}^{i+j}(U_{1},M)\Rightarrow
\text{H}^{i+j}(U(\mathbb{F}_p),M).$$
This shows that if $M$ in $\text{mod}(B)$ then $c_{U_1}(M)\geq
c_{U(\mathbb{F}_p)}(M)$. Moreover, one obtains the following theorem. 

\begin{thm}\cite[Thm. 3.4]{LN1}
Let $M$ be in $\operatorname{mod}(G)$. Then $c_{G(\mathbb{F}_p)}(M)\leq
\frac{1}{2}c_{G_1}(M)$.
\end{thm}

\begin{proof} (Sketch) First $c_{G(\mathbb{F}_p)}(M)=c_{U(\mathbb{F}_p)}(M)$ 
because $U({\mathbb F}_{p})$ is a Sylow subgroup of $G({\mathbb F}_{p})$. Next $c_{U_1}(M)=c_{B_{1}}(M)=\frac{1}{2}c_{G_1}(M)$ by using a result of Spaltenstein 
which says that the intersection of a $G$-orbit inside the nilpotent cone with 
${\mathfrak n}$ is a union of $B$-stable sets having one-half the dimension of 
the original orbit. Finally, one can apply the inequality $c_{U_1}(M)\geq
c_{U(\mathbb{F}_p)}(M)$. 
\end{proof} 

As a corollary, one can obtain a proof of the Parshall Conjecture. 

\begin{cor} Let $M$ be in $\operatorname{mod}(G)$. If $M$ is a projective 
$G_1$-module then $M$ is a projective $kG(\mathbb{F}_p)$-module.
\end{cor}

\begin{proof} If $M$ is projective as a  $G_1$-module then $c_{G_1}(M)=0$. Therefore, 
$c_{G(\mathbb{F}_p)}(M)=0$ by using the complexity bound, thus $M$ is
a projective $kG(\mathbb{F}_p)$-module.
\end{proof} 

Let us consider these results in the context of the representation type question in 
Section 4.1. Since $kSL_{2}({\mathbb F}_{p})$ has finite representation type we 
have $c_{SL_2(\mathbb{F}_p)}(k)=1$. Therefore, $c_{(SL_2)_1}(k)\geq 2$ which 
demonstrates immediately that $(SL_2)_1$ has infinite representation type because 
the syzygies $\Omega^{n}(k)$ are indecomposable modules and must grow in dimension at least linearly in $n$.   

\subsection{Partial Converse to the Parshall Conjecture: } 

A natural question to ask is whether the converse to the Parshall 
Conjecture holds, that is if $M$ is in $\text{mod}(G)$, and $M$ is projective 
over $kG(\mathbb{F}_p)$, is it projective over $G_1$?

The immediate answer is no. Let $M=St^{(1)}=L((p-1)\rho)^{(1)}$. Then 
$M$ is projective as $G(\mathbb{F}_p)$-module, and $M\simeq \oplus k$ as
$G_1$-module. Let ${\mathcal C}_p$ be the full 
subcategory of $G$-modules whose composition factors have highest weight 
$\lambda$, satisfying $\langle \lambda, \widetilde{\alpha}^{\vee}\rangle \leq p(p-1)$ 
where $\widetilde{\alpha}$ is the highest root. The converse is valid as long as we 
restrict the collection of $G$-modules under consideration.  

\begin{thm} \cite[Thm. 4.4]{LN1}
Let $M\in {\mathcal C}_p$. If $M$ is projective over
$kG(\mathbb{F}_p)$ then $M$ is projective over $G_1$.
\end{thm}
 
Friedlander and Parshall proved that can one check projectivity over $G_{1}$ for 
modules in $\text{mod}(G)$ by considering the restriction to the one-dimensional 
Lie algebra spanned by a long root vector. 

\begin{thm} \cite[(2.4b) Prop.]{FP1} Let $\alpha\in \Phi$ be long root and $M$ be in $\operatorname{mod}
(G)$. Then $M$ is
projective over $G_1$ if and only if $M\mid_{u(\langle x_{\alpha})\rangle)}$ is projective.
\end{thm}

We can now apply Friedlander and Parshall's result above along with the validity 
of the Parshall Conjecture to prove a projectivity criterion over 
$G({\mathbb F}_{p})$ for modules in ${\mathcal C}_{p}$. 

\begin{cor}\cite[Thm 4.5]{LN1}
Let $M$ be in ${\mathcal C}_p$. The following are equivalent:
\begin{itemize}
\item[(a)] $M$ is projective over $kG(\mathbb{F}_p)$; 
\item[(b)] $M\mid_{x_{\alpha}(\mathbb{F}_p)}$ is projective for all 
$\alpha\in \Phi$; 
\item[(c)] $M\mid_{x_{\beta}(\mathbb{F}_p)}$ is projective for
$\beta\in \Phi$ long.
\end{itemize}
\end{cor}

\begin{proof} The implications $(a)\Rightarrow (b) \Rightarrow (c)$ are clear. 
We need to show that $(c)\Rightarrow (a)$.
Suppose that $M\mid_{x_{\beta}(\mathbb{F}_p)}$ is projective. 
Then $M\mid_{SL_2(\mathbb{F}_p)_{\beta}}$ is~projective where $SL_2(\mathbb{F}_p)_{\beta}$ is the semisimple part of the subgroup generated by 
$x_{\beta}({\mathbb F}_{p})$ and $x_{-\beta}({\mathbb F}_{p})$. By the partial 
converse to the Parshall Conjecture, $M\mid_{{u(\mathfrak{sl}_2)_{\beta}}}$ 
is projective. Therefore, $M\mid_{u(\langle x_{\beta}\rangle)}$ 
is~projective. Friedlander and Parshall's result implies that $M$ is 
projective as $G_1$-module. The Parshall Conjecture can be employed to conclude that 
$M$ is projective as $G(\mathbb{F}_p)$-module.
\end{proof}

\subsection{A Map Between Support Varieties: }

Parshall also posed the problem in his Arcata Conference Proceedings article to 
find a relationship between the support varieties for $G_1$ and $G(\mathbb{F}_p)$. 
The results in this section are due to the Carlson, Lin and the author \cite{CLN}. 

Let ${\mathcal V}_{G_1}(M)$ denote the support variety of $M$ in 
$\text{mod}(G_{1})$. Recall that  
$${\mathcal V}_{G_1}(M)=\{x\in {\mathfrak g}:\ x^{[p]}=0,\ M|_{u(\langle 
x\rangle)}\ \text{is not free}\}\cup\{0\}\subseteq
\mathcal{N}_1(g)\subseteq\mathcal{N}.$$ 
We will assume that $p$ is good (i.e., if $\beta\in \Phi^{+}$ with  
$\beta=\sum\limits_{\alpha\in\Delta}n_{\alpha}\alpha$ then 
$p\nmid n_{\alpha}$ for any $\alpha$). For example, $\Phi=E_6$, $p$ is good 
except when $p=2,3$. The following theorem provides a description of ${\mathcal V}_{G_{1}}(k)\cong 
{\mathcal N}_{1}({\mathfrak g})$ when $p$ is good. 

\begin{thm}\cite[(6.2.1) Thm.]{NPV} Let $p$ be good. Then there exist 
$J\subseteq\Delta$ such that $\mathcal{N}_1({\mathfrak g})=G\cdot
{\mathfrak u}_J$, where ${\mathfrak g}={\mathfrak u}_J\oplus 
{\mathfrak l}_J\oplus {\mathfrak u}_J^+$~ where ${\mathfrak l}_J$
is the corresponding Levi subalgebra. Furthermore, the nilpotence 
degree of ${\mathfrak u}_J<p$. 
\end{thm}

Let ${\mathcal U}$ be the unipotent variety, which is the set of
unipotent element in $G$, and set 
$${\mathcal U}_1=\{u\in{\mathcal U}:\ u^p=1\}.$$
When $p$ is good, ${\mathcal U}_1=G\cdot U_J$ where $U_J$ is the unipotent 
radical of the parabolic subgroup associated to $J$ (as in the preceding theorem). 
By using variations on the exponential map, Seitz was able to prove the 
following theorem. 

\begin{thm} \cite[Prop. 5.1]{Sei} There exists a 
$P_{J}$-equivariant isomorphism of algebraic k-varieties:
$$\operatorname{exp}:{\mathfrak u}_J\rightarrow {\mathcal U}_J$$
which is defined over $\mathbb{Z}_{(p)}$.
\end{thm}

Under suitable normality conditions, we were able to extend Seitz's map to a 
$G$-equivariant isomorphism of varieties. 

\begin{thm} \cite[Thm. 3]{CLN}
Let $\mathcal{N}_1(g)$ be normal. Then exp extends to a
G-equivariant isomorphism:
$$\operatorname{exp}:G\cdot {\mathfrak u}_J\rightarrow G\cdot U_J.$$
\end{thm}

From now on let us define  
$$\text{log}:{\mathcal U}_1\rightarrow \mathcal{N}_1({\mathfrak g})$$
as the inverse of $\text{exp}$. For $M$ a 
$G({\mathbb F}_{p})$-module, let ${\mathcal V}_{G(\mathbb{F}_p)}(M)$ denote the 
support variety of $M$. One can describe the support variety of the trivial 
module as 
$${\mathcal V}_{G(\mathbb{F}_p)}(k)=\lim\limits_{\overrightarrow{E}}{\mathcal V}_E(k),$$
where $E$ ranges over the elementary abelian $p$-subgroups. Now one can identify 
${\mathcal V}_E(k)$ with a rank variety ${\mathcal V}_E^{rank}(k)$ which allows one to 
construct a map ${\mathcal V}_E(k)\hookrightarrow {\mathcal U}_1$. By putting all these 
maps together we have  
$${\mathcal V}_{G(\mathbb{F}_p)}(k)\hookrightarrow {\mathcal U}_1/G(\mathbb{F}_p)\rightarrow 
\mathcal{N}_1({\mathfrak g})/G(\mathbb{F}_p)\cong V_{G_1}(k)/G(\mathbb{F}_p).$$
Therefore, we have defined a map of support varieties: 
$$\psi:{\mathcal V}_{G(\mathbb{F}_p)}(k)\rightarrow {\mathcal V}_{G_1}(k)/G(\mathbb{F}_p).$$

Now define the category ${\mathcal D}$ as follows.  Let $l=\text{rank}(\Phi)$ and
$(b_{ij})=(\langle\alpha_i,\alpha_j^{\vee}\rangle)^{-1}$.  Let
${\mathcal D}$ be the full subcategory of G-module whose composition factors
have high weight satisfying 
$$\sum_{i=1}^{l}\sum_{j=1}^{l}\langle\lambda,\alpha_i^{\vee}\rangle b_{ij}<\frac{p(p-1)}{2}.$$

The following theorem gives an explicit description of the image of $\psi$ on 
${\mathcal V}_{G(\mathbb{F}_p)}(M)$. 

\begin{thm} \cite[Cor. 1,2]{CLN} Let $G$ be a simple algebra group scheme defined over $\mathbb{F}_p$. 
Moreover, let $\mathcal{N}_1(g)$ be normal and $M\in {\mathcal D}$,then
$$\psi({\mathcal V}_{G(\mathbb{F}_p)}(M))\subseteq {\mathcal V}_{G_1}(M)/G(\mathbb{F}_p).$$
More precisely:
$$\psi({\mathcal V}_{G(\mathbb{F}_p)}(M))=\{x\in\mathcal{N}_1^{\mathbb{F}_p}:\ 
x^{[p]}=0,\ M|_{u(\langle x\rangle)} \text{is not free}\}\cup\{0\}$$
where $\mathcal{N}_1^{\mathbb{F}_p}$ is the set of ${\mathbb F}_{p}$-expressible elements (cf. 
\cite[Section 3]{CLN}). 

\end{thm}

\subsection{Applications: } The results in the preceding section can 
be used to understand projectivity for finite groups of Lie type. 

\begin{thm}
Let G be a simple algebraic group, $M$ be in ${\mathcal C}_p$ which is not
projective over $kG(\mathbb{F}_p)$. Then
\begin{itemize} 
\item[(a)] If $M$ is indecomposable in $\operatorname{mod}(kG)$ and 
the rank of $\Phi$ is greater than or equal to $2$ then
$\dim {\mathcal V}_{G(\mathbb{F}_p)}(M)\geq 2$;
\item[(b)] If $\Phi=A_n$ then $\dim {\mathcal V}_{GL_n(\mathbb{F}_p)}(M)\geq n$.
\end{itemize} 
\end{thm}

We can also determine the dimension of support varieties for simple 
modules for rank 2 groups. 

\begin{thm} Let G be a simple algebraic group where $\Phi$ has 
rank 2 (i.e., $\Phi=A_2$ ($p>2$), $B_2$ ($p>7$), $G_2$ ($p>19$), and 
let $\lambda\in X_1(T)$. 
\begin{itemize}
\item[(a)] If $\lambda=(p-1)\rho$ then ${\mathcal V}_{G(\mathbb{F}_p)}(L(\lambda))=\{0\}$.  
\item[(b)] If $\lambda\neq(p-1)\rho$ then 
$$\dim {\mathcal V}_{G(\mathbb{F}_p)}(L(\lambda))=
\begin{cases}
2 & \Phi=A_2\\
3 & \Phi=B_2,G_2.
\end{cases}
$$
\end{itemize}
\end{thm}

\subsection{Generalizations: } Friedlander \cite{F2} has recently found an appropriate 
generalization of many of the aforementioned results for $G({\mathbb F}_{q})$ where 
$r\geq 1$. The idea involves using base changes on the Lie algebra level and to 
view $g_{\mathbb{F}_q}$ as an $\mathbb{F}_p$-Lie algebra. Set 
$A_q=u(g_{\mathbb{F}_q}\otimes_{\mathbb{F}_p}k)$. In particular, Friedlander proves 
a generalization of the Parshall Conjecture. 

\begin{thm}\cite[Cor. 4.4]{F2} Let $M\in Mod(G)$. 
If $M\downarrow_{A_q}$ is projective then $M|_{G(\mathbb{F}_q)}$ is projective.
\end{thm}

Furthermore, by using the $\Pi$-point theory developed by Friedlander and Pevtsova 
\cite{FPe}, he constructs a map 
$$\psi:{\mathcal V}_{G(\mathbb{F}_q)}(k)\hookrightarrow {\mathcal V}_{A_q}(k)/G(\mathbb{F}_q)$$
(cf. \cite[Corollary 3.6]{F2}).

\section{Relating Cohomology} 

\subsection{Spectral Sequences via Truncation: } We will next investigate questions 
which involve relating extensions in the three categories $\text{Mod}(G({\mathbb F}_{q})$, $\text{Mod}(G)$, 
and $\text{Mod}(G_{r})$. The results were obtained in \cite{BNP1,BNP2,BNP3} with generalizations 
to the twisted cases in \cite{BNP5,BNP6}. 

The basic idea is to first construct spectral sequences which 
connect the cohomology theories of $\text{Mod}(G({\mathbb F}_{q}))$ and $\text{Mod}(G)$. 
The basic problem is that $\text{Mod}(G)$ does not have enough projective objects. Therefore, 
it is better to work with ``truncated categories'' obtained via saturated sets of weights. These 
truncated categories are highest weight categories and Morita equivalent to the module category 
for some quasi-hereditary algebra (cf. \cite{CPS}). The 
cohomology for these quasi-hereditary algebras provides a better approximation of the homological behavior 
in $\text{Mod}(G({\mathbb F}_{q}))$. Once we provide a connection to the $G$-category, we apply 
the Lyndon-Hochschild-Serre (LHS) spectral sequence and use information about the cohomology in 
$\text{Mod}(G_{r})$, in addition to the ``geometry of the flag variety $G/B$'' to obtain results 
about the cohomology for the finite Chevalley groups. 

Let $\Pi=\{\lambda\in X(T)_+: 
\langle \lambda+\rho,\alpha_0^{\vee} \rangle \leq 2p^r \langle \rho,\alpha_0^{\vee} 
\rangle\}$ and $\mathcal{C}$ be the full subcategory of modules in $\text{Mod}(G)$ 
whose composition factors have highest weight in $\Pi$. Let 
${\mathcal T}:\text{Mod}(G)\rightarrow\mathcal{C}$ be the functor 
defined where ${\mathcal T}(M)$ is the largest submodule of M contained in $\mathcal{C}$. 
The functor ${\mathcal T}$ is left exact. 

Now define the functor ${\mathcal G}:\text{Mod}(G(\mathbb{F}_q))\rightarrow\mathcal{C}$ as 
$${\mathcal G}(N)={\mathcal T}(\text{ind}_{G(\mathbb{F}_q)}^G N)$$ 
where $N\in \text{Mod}(G({\mathbb F}_{q}))$. Since $G/G({\mathbb F}_{q})$ is affine, the 
induction functor is exact. Therefore, ${\mathcal G}$ is a left exact functor and admits higher 
right derived functors $R^{j}{\mathcal G}(-)$ for $j\geq 0$. 

The following theorem provides a spectral sequence which relates extensions 
in $\text{Mod}(G({\mathbb F}_{q})$ and ${\mathcal C}$. 

\begin{thm} Let $M$ in $\mathcal{C}$ and $N$ in $\operatorname{Mod}(G(\mathbb{F}_q))$. Then there exists a first 
quadrant spectral sequence: 
$$E_2^{i,j}=\operatorname{Ext}_{\mathcal{C}}^i(M,R^j{\mathcal G}(N))\Rightarrow 
\operatorname{Ext}_{G(\mathbb{F}_q)}^{i+j}(M,N).$$
\end{thm} 

We note that since ${\mathcal C}$ is obtained from a saturated set of weight, we have 
$$\text{Ext}^{i}_{\mathcal{C}}(M_{1},M_{2})=\text{Ext}_G^i(M_{1},M_{2})$$
for all $M_{1},M_{2}$ in ${\mathcal C}$ and $i\geq 0$. One can apply this spectral sequence along 
with the lifting property of projective $G_{r}$-modules to obtain the following corollary. 

\begin{cor} Let $p\geq 2(h-1)$ with   $\lambda,\mu\in X_r(T)$
\begin{itemize} 
\item[(a)] $\operatorname{Ext}_{G(\mathbb{F}_q)}^1(L(\lambda),L(\mu))\cong 
\operatorname{Ext}_G^1(L(\lambda),{\mathcal G}(L(\mu)))$;
\item[(b)] $\operatorname{Ext}_G^2(L(\lambda),{\mathcal G}(L(\mu)))\hookrightarrow
\operatorname{Ext}_{G(\mathbb{F}_q)}^2(L(\lambda),L(\mu))$.
\end{itemize} 
\end{cor}
\begin{proof} Let $M=L(\lambda)$, $N=L(\mu)$ in the spectral sequence. The spectral sequence yields 
a five term exact sequence: 
$$0\rightarrow E_2^{1,0}\rightarrow E_1\rightarrow E_2^{0,1}\rightarrow E_2^{2,0}\rightarrow E_2.$$
In order to prove the result we need to show that $E_2^{0,1}=0$. It suffices to prove that 
$R^{1}{\mathcal G}(L(\mu))$ has no composition factors isomorphic to $L(\lambda)$. When $p\geq 2(h-1)$, the 
projective cover $P(\lambda)$ of $L(\lambda)$ in ${\mathcal C}$ is isomorphic to $Q_{r}(\lambda)$. This 
fact uses the lifting property of projective $G_{r}$-modules. Moreover, one can apply the 
spectral sequence with $M=P(\lambda)$ and $N=L(\mu)$. The spectral sequence then collapses and yields: 
$$\text{Hom}_{\mathcal C}(P(\lambda),R^1{\mathcal G}(L(\mu)))\cong \text{Ext}_
{G(\mathbb{F}_q)}^1(P(\lambda),L(\mu)).$$

By putting all of this information together, we have 
\begin{eqnarray*}
[R^1{\mathcal G}(L(\mu)):L(\lambda)]&=& \dim \text{Hom}_{\mathcal C}(P(\lambda),R^1{\mathcal G}(L(\mu)))\\
&=&\dim \text{Ext}_{G(\mathbb{F}_q)}^1(P(\lambda),L(\mu))\\
&=&\dim \text{Ext}_{G(\mathbb{F}_q)}^1(Q_r(\lambda),L(\mu))\\
&=& 0.
\end{eqnarray*}
\end{proof} 

\subsection{Extensions between simple modules: } Our knowledge about extensions involving two simple 
modules in $\text{Mod}(G)$, $\text{Mod}(G_{r})$ and $\text{Mod}(G({\mathbb F}_{q})$ is very minimal. 
We can state this in terms of three open problems: 
\vskip .25cm 
\noindent 
(5.2.1) Determine $\text{Ext}_G^j(L(\sigma_1),L(\sigma_2))$ for $\sigma_1,\sigma_2\in X(T)_+$, $j\geq 0$.
\vskip .15cm 
\noindent 
(5.2.2) Determine $\text{Ext}_{G_r}^j(L(\sigma_1),L(\sigma_2))$ for $\sigma_1,\sigma_2\in X_r(T)$, $j\geq 0$.
\vskip .15cm 
\noindent 
(5.2.3) Determine $\text{Ext}_{G(\mathbb{F}_q)}^j(L(\sigma_1),L(\sigma_2))$ for $\sigma_1,\sigma_2\in X_r(T)$, $j\geq 0$.
\vskip .25cm 

The Lusztig Conjecture predicts the characters of the simple $G$-modules when $p\geq h$. If the 
Lusztig Conjecture holds then (5.2.1) and (5.2.2) can be solved for regular weights. Moreover, 
if there are certain vanishing/non-vanishing conditions that hold when $j=1$ for the extension groups in 
(5.2.1) and (5.2.2) then the Lusztig Conjecture is valid. 

It is interesting to consider the case when $L(\sigma_{1})\cong L(\sigma_{2})\cong k$. 
Then $\text{H}^{j}(G,k)=0$ for $j>0$ and is isomorphic to $k$ when $j=0$. For the Frobenius kernels, 
$\text{H}^{2\bullet}(G_{1},k)=k[{\mathcal N}]$ where ${\mathcal N}$ is the nilpotent cone, 
and $\text{H}^{2\bullet+1}(G_{1},k)=0$ when $p>h$ (cf. \cite{AJ, FP2}). The 
cohomology for $G_{r}$ is not known in general for $r>1$. Finally, very little is known about the cohomology groups 
$\text{H}^{\bullet}(G({\mathbb F}_{q}),k)$. We will investigate questions about these groups later in Section 6.  

In the case when $\sigma=\sigma_{1}=\sigma_{2}$ and $j=1$ one is looking at the case of 
self-extensions. By standard arguments \cite[II 2.12]{Jan1}, one can show that 
$\text{Ext}_G^1(L(\sigma),L(\sigma))=0$ for $\sigma \in X(T)_+$. In 1984, Andersen \cite[Theorem 4.5]{And1} 
proved the following result about self extensions for $G_{r}$:

\begin{thm} Suppose $\Phi\neq C_n,$ when $p=2$. Then 
$\operatorname{Ext}_{G_r}^1(L(\sigma),L(\sigma))=0$ for all $\sigma\in X_r(T)$. 
\end{thm}

Andersen used information about the structure of the higher line bundle cohomology groups in his calculation. 
The use of these geometric techniques will be a dominant theme in our examination of  self-extensions and 
cohomology for finite Chevalley groups 

\subsection{$\text{Ext}^{1}$-formulas for simple $G({\mathbb F}_{q})$-modules: }We will now run through a 
series of steps which will lead us to produce $\text{Ext}^{1}$-formulas for finite Chevalley groups. 
\vskip .25cm 
\noindent 
(1) Let $\lambda,\mu\in X_r(T)$, then
$$\text{Ext}_{G({\mathbb F}_q)}^1(L(\lambda),L(\mu))\cong \text{Ext}_G^1
(L(\lambda),L(\mu)\otimes {\mathcal G}(k))$$ 
for all $\lambda, \mu\in X_{r}(T)$ \cite[Theorem 2.2]{BNP3}.
\vskip .25cm 
\noindent 
(2) We can explicitly describe ${\mathcal G}(k)$ as a $G$-module. Set 
$\Gamma_{2h-1}=\{\nu \in X(T)_+: \langle \nu ,\alpha_0^{\vee} \rangle <2h-1\}$. 
Then a miracle happens! For $p\geq 3(h-1)$, ${\mathcal G}(k)$ is semisimple \cite[Theorem 7.4]{BNP1}. 
Moreover, 
\begin{eqnarray*}
{\mathcal G}(k)&\cong &\bigoplus_{\nu\in \Gamma_{2h-1}}L(\nu)\otimes[L(\nu)^{(r)}]^*\\
               &\cong &\bigoplus_{\nu \in \Gamma_{2h-1}} L(v)\otimes L(v^{*})^{(r)}.
\end{eqnarray*}
\vskip .25cm 
(3) Set $\Gamma=\{\nu \in X(T)_+: \langle \nu,\alpha_0^{\vee} \rangle < h-1\}$. By 
using the decomposition of ${\mathcal G}(k)$ along with some additional information about 
$G$-extensions, we obtain the following formula. 

\begin{thm} Let $p\geq 3(h-1)$ and $\lambda,\mu\in X_r(T)$. Then 
$$\operatorname{Ext}_{G(\mathbb{F}_q)}^1(L(\lambda),L(\mu))\cong 
\bigoplus_{\nu \in\Gamma} \operatorname{Ext}_G^1(L(\lambda)\otimes L(\nu)^{(r)},L(\mu)\otimes L(\nu)).$$
\end{thm}
\vskip .25cm 
We can now refine the aforementioned $\text{Ext}^{1}$-formula when $r\geq 2$ and $r=1$. First 
let $r\geq2$. For $\sigma\in X_r(T)$ with $\sigma=\sigma_0+\sigma_1p+\cdots+\sigma_{r-2}p^{r-2}+\sigma_{r-1}p^{r-1}$ set 
$\widehat{\sigma}=\sigma_0+\sigma_1 p+\cdots+\sigma_{r-2}p^{r-2}$. We can apply the LHS spectral sequence twice 
with $G_{r-1}\unlhd G$, and $G_1\unlhd G$ to obtain the following theorem. 

\begin{thm}\label{selfextr=2} Let $p\geq 3(h-1)$, $\lambda,\mu\in X_r(T)$, $r\geq 2$. Then
$$\operatorname{Ext}_{G(\mathbb{F}_q)}^1(L(\lambda),L(\mu))=\operatorname{Ext}_{G}^1(L(\lambda),L(\mu))\oplus R$$
where $R=\bigoplus\limits_{\nu \in
\Gamma-\{0\}}\operatorname{Hom}_G(L(\nu),\operatorname{Ext}_{G_1}^1(L(\lambda_{r-1}),L(\mu_{r-1})))
\otimes\operatorname{Hom}_G(L(\widehat{\lambda}),L(\widehat{\mu})\otimes
L(\nu))$. 
\end{thm} 

One can also produce a formula when $r=1$. In these formulas one should observe how the information from the 
Frobenius kernels is playing a role in the determination of cohomology for the finite Chevalley groups. 

\begin{thm} Let $p\geq 3(h-1)$ and $\lambda, \mu\in X_{1}(T)$.
Then
\begin{equation*}
\operatorname{Ext}^{1}_{G({\mathbb F}_{p})}(L(\lambda),L(\mu))
\cong  \operatorname{Ext}^{1}_{G}(L(\lambda),L(\mu))\oplus R
\end{equation*}
where
\begin{equation*}
R =  \bigoplus_{\nu\in \Gamma-\{0\}}
\operatorname{Hom}_{G}(L(\nu),\operatorname{Ext}^{1}_{G_{1}}
(L(\lambda),L(\mu)\otimes L(\nu))^{(-1)}).
\end{equation*}
\end{thm}

\subsection{Self extensions: } For $r\geq 2$, the groups $G({\mathbb F}_{q})$ do not admit self extensions 
when $p\geq 3(h-1)$. 

\begin{thm} Let $p\geq 3(h-1)$, $r\geq 2$ and $\lambda\in X_{r}(T)$. Then 
$\operatorname{Ext}^{1}_{G({\mathbb F}_{q})}(L(\lambda),L(\lambda))=0$.
\end{thm}

\begin{proof} As we have seen $\Ext_{G}^{1}(L(\lambda),L(\lambda)) = 0$. Therefore, we need to 
prove that $R=0$ in Theorem~\ref{selfextr=2}. But, this follows by applying Andersen's result 
on self-extensions for $G_{1}$. 
\end{proof}

For $r=1$, Humphreys proved the existence of self extensions when the root system is of type $C_{n}$ for 
general $p$. One can use a more detailed analysis using the cohomology of line bundles to obtain the following 
result (cf. \cite[Section 4]{BNP3}).

\begin{thm} Let  $p \geq 3(h-1)$
and $\lambda \in X_1(T)$. If either
\begin{itemize}
\item[(a)] $G$ does not have underlying root system of type $A_1$ or 
$C_n$
or
\item[(b)] $\langle \lambda,\alpha_n^{\vee}\rangle \neq \frac{p - 2 - c}{2}$,
where $\alpha_n$ is the
unique long simple root  and $c$ is odd with $0 < |c| \leq h - 1$,
\end{itemize}
then $\Ext_{G({\mathbb F}_{p})}^{1}(L(\lambda), L(\lambda)) = 0.$
\end{thm}

\subsection{Applications: } Our results involving $\text{Ext}^{1}$ for finite Chevalley group 
enabled us to address several open questions. 
\vskip .15cm 
\noindent 
(1) Smith asked the following question (Question 5) in his 1985 paper \cite{Sm}: 
\vskip .15cm 
\noindent 
{\bf Question 5: } Let $V$ be a $G(\mathbb{F}_q)$-module satisfying
Hypothesis A with all composition factors  isomorphic. Is $V$
completely reducible?

Our self extension results stated in Section 5.4 answer Smith's question independent of ``Hypothesis A''.
\vskip .25cm 
\noindent 
(2) Cline, Parshall, Scott, and vanderKallen's \cite{CPSvdK} famous result on 
rational and generic cohomology can be stated as follows. Let $V$ in $\text{Mod}(G)$. For 
fixed $i\geq 0$, consider the restriction map in cohomology: 
$$\text{H}^i(G,V^{(s)})\xrightarrow[\text{res}]{}\text{H}^i(G(\mathbb{F}_q),V^{(s)}).$$ 
When $r$ and $s$ are sufficiently large, the map $res$ is an isomorphism.

We can see the phenomenon in our $\text{Ext}^{1}$-formulas. Let $p\geq 3(h-1)$. 
Fix $\lambda,\mu\in X_{r}(T)$. As $r$ get large $\lambda_{r-1}=\mu_{r-1}=0$, 
and $\text{Ext}_{G_1}^1(k,k)=0$. Therefore, by Theorem~\ref{selfextr=2}, $R=0$ and 
$$\text{Ext}_{G}^1(L(\lambda),L(\mu))\xrightarrow[\text{res}]{}\text{Ext}_{G(\mathbb{F}_q)}^1(L(\lambda),L(\mu))$$
is an isomorphism. Observe that in this instance no Frobenius twists are necessary. 
\vskip .25cm 
\noindent
(3) Jantzen \cite{Jan4} proved the following result which insured that the 
restriction map in cohomology is an isomorphism given upper bounds on the highest 
weight of the composition factors of $V$. 

\begin{thm} Let $V\in \operatorname{Mod}(G)$ where the composition factors $L(\mu)$ satisfy 
\[
\langle \mu,\alpha_0^{\vee} \rangle \leq\left\{ \begin{array}{ll}
p^r-3p^{r-1}-3 & \Phi=G_2\\
p^r-2p^{r-1}-2 & \Phi\neq G_2.
\end{array}
\right.
\]
Then $\operatorname{res}$ is an isomorphism.
\end{thm}
For large primes, Andersen \cite{And2} gives uniform conditions on the high weight on the 
composition factors to insure that the restriction map is an isomorphism. 
 
\begin{thm} Let $p\geq 3(h-1)$ and suppose that the composition factors of
$V$ satisfy
$$\langle \mu,\alpha_0^{\vee} \rangle \leq p^r-p^{r-1}-2\quad (\Phi\neq A_1).$$
Then $\operatorname{res}$ is an isomorphism.
\end{thm} 

Our results \cite[Theorem 4.8(A)]{BNP2} enabled us to give the best possible bounds to insure that the restriction 
map is an isomorphism when $r\geq 2$. The optimal bounds when $r=1$ are still unknown. 

\begin{thm} Let G be a simple, simply connected algebraic group with $p\geq 3(h-1)$, 
$r\geq2 $. If V has composition factors $L(\mu)$ satisfying
$$
\langle \mu,\alpha_0^{\vee} \rangle \leq \begin{cases} 
p^r-2p^{r-1} & \Phi=A_1\\
p^r-p^{r-1} & \Phi=A_n\\
p^r & \Phi=B_n,~C_n,~D_n,~E_6,~E_7\\
2p^r-p^{r-1}+1  & \Phi=E_8,~F_4,~G_2.
\end{cases}  
$$
Then $\operatorname{res}$ is an isomorphism.
\end{thm}

\section{Computing Cohomology for Finite Groups of Lie Type} 

\subsection{Vanishing Ranges: } In 2005 there was a Conference on the Cohomology 
of Finite Groups at Oberwolfach. At the beginning of Eric Friedlander's talk he 
started by stating that we know very little about the cohomology $\text{H}^{\bullet}(G({\mathbb F}_{q}),k)$ 
where $\text{char }k=p>0$ other than Quillen's description of the maximal ideal spectrum. Friedlander 
went on to say that it would be nice to know when the cohomology starts (i.e., in which degree the 
first non-trivial cohomology lives). There are two main aspects to this problem. 
\vskip .25cm 
\noindent 
(6.1.1) Vanishing Ranges: Locating $D>0$ such that 
the cohomology groups $\opH^i(\gfpr,k)=0$ for $0< i < D$.
\vskip .15cm 
\noindent 
(6.1.2) Determining the First Non-Trivial Cohomology Class: Find a $D$ such that $\opH^i(\gfpr,k)=0$ 
for $0<i < D$ and $\opH^{D}(\gfpr,k)\neq 0$. A $D$ satisfying this property will be called a {\em sharp bound}. 
\vskip .25cm 
Quillen computed the cohomology of $G({\mathbb F}_{q})$ in the non-describing characterisitic and provided 
a vanishing range (6.1.1) in the case of $GL_{n}({\mathbb F}_{q})$ in the describing characteristic 
case (cf. \cite{Q3}). Later Friedlander \cite{F1} and Hiller \cite{H} discovered vanishing ranges for 
other groups of Lie type. 

The aim of this section is to present some recent work of the author with Bendel and Pillen. We 
solve (6.1.2) for groups of types $A_{n}$ and $C_{n}$ where $p>h$ and $r=1$. In order to obtain these 
results we first employ a variation 
on the ideas presented in Section 5 by investigating the induction ${\mathcal G}_{r}(k)=\text{ind}_{G({\mathbb F}_{q})}
^{G} k$. We indicate that ${\mathcal G}_{r}(k)$ has a natural filtration as a $G\times G$-module. This 
allows us to reduce our problem to looking at sections of the filtration 
(i.e., modules of the form $H^{0}(\lambda)\otimes 
H^{0}(\lambda^{*})^{(r)}$, $\lambda\in X(T)_{+}$). The next step is to apply the LHS spectral sequence in the case when 
$r=1$. The spectral sequence collapses in the case when $p>h$ by using results due to Kumar, Lauritzen and Thompsen 
\cite{KLT}. This in turn provides a method for giving an upper bound for the cohomology $\opH^{\bullet}(G({\mathbb F}_{p}),k)$ 
using the combinatorics of the nilpotent cone via Kostant's Partition Functions. The steps are presented in the following 
diagram. The results and details in this section can be found in \cite{BNP8,BNP9}. 

\begin{picture}(470,110)(0,0)
\put(0,50){$\text{H}^i(G({\mathbb F}_{q}),k)$}
\put(70,50){$ \rightarrow$}
\put(55,82){$ \text{ Induction}$}
\put(59,70){$ \text{ Functor}$}
\put(100,50){$\opH^i(G,\cgr)$}
\put(60,10){$\text{H}^i(G,H^0(\lambda) \otimes H(\lambda^*)^{(r)})$}
\put(195,10){$ \rightarrow$}
\put(130,30){$ \downarrow$}
\put(165,-10){$ \text{ LHS Spectral}$}
\put(150,30){$ \text{ Filtrations}$}
\put(172,-22){$ \text{ Sequences}$}
\put(220,10){$\text{H}^i(G_1,H^0(\lambda))$}
\put(300,10){$ \rightarrow$}
\put(268,-10){$ \text{Kostant Partition}$}
\put(280,-22){$ \text{ Functions}$}
\put(328,10){Root Combinatorics.}
\end{picture}

\vskip 2cm 

\subsection{Induction and Filtrations: } Let ${\mathcal G}_{r}=\text{ind}_{G({\mathbb F}_{q})}^{G}(-)$. This 
functor is exact and all its higher right derived functors vanish. Therefore, we can apply Frobenius 
reciprocity to obtain the following isomorphism of extension groups. 

\begin{prop}\label{Frobreciso} Let $M, N$ be in $\operatorname{Mod}(G)$. Then, for all $i \geq 0$,
$$
\Ext^i_{\gfpr}(M,N) \cong \Ext_G^i(M,N\otimes\cgr).
$$ 
\end{prop}

As a $G\times G$-module the coordinate algebra $k[G]$ has a filtration with sections of the form 
$H^{0}(\lambda)\otimes H^{0}(\lambda^{*})$ where $\lambda\in X(T)_{+}$. Each section of this 
form appears exactly one time. By using the Lang map, we were able to show that 
${\mathcal G}_{r}(k)=k[G/G({\mathbb F}_{q})]$ admits a natural filtration as a $G\times G$-module. 

\begin{prop} The $G\times G$-module $\cgr$
has a filtration with factors of the form $H^0(\lambda)\otimes H^0(\lambda^*)^{(r)}$ 
with multiplicity one for each $\lambda \in X(T)_+$.
\end{prop} 

\subsection{} The existence of a filtration on ${\mathcal G}_{r}(k)$ with sections of 
the form $H^0(\lambda)\otimes H^0(\lambda^*)^{(r)}$, $\lambda\in X(T)_{+}$, in addition 
to the isomorphism in Proposition~\ref{Frobreciso} allows us to deduce the next 
result on determining vanishing ranges. 

\begin{prop}\label{vancoho} Let $m$ be the least positive integer such that there exists
$\lambda \in X(T)_{+}$ with $\opH^m(G,H^0(\lambda)\otimes H^0(\lambda^*)^{(r)}) \neq 0$.
Then $\opH^i(\gfpr,k) \cong \opH^i(G,\cgr) = 0$ for $0 < i < m$.  
\end{prop}

A more detailed analysis is necessary to address (6.1.2) using this aformentioned filtration 
(cf. \cite[Section 2.7]{BNP7}). Our results from this section are summarized below. 

\begin{thm}\label{firstclass}  Let $m$ be the least positive integer such that there exists
$\nu \in X(T)_{+}$ with $\opH^m(G,H^0(\nu)\otimes H^0(\nu^*)^{(r)}) \neq 0$. 
Let $\lambda \in X(T)_+$ be such that 
$\opH^m(G,H^0(\lambda)\otimes H^0(\lambda^*)^{(r)}) \neq 0$.  
Suppose $\opH^{m+1}(G,H^0(\nu)\otimes H^0(\nu^*)^{(r)}) = 0$ for all $\nu < \lambda$ that are linked to $\lambda$.
Then
\begin{itemize}
\item[(i)] $\opH^i(\gfpr,k) = 0$ for $0 < i < m$;
\item[(ii)] $\opH^m(\gfpr,k) \neq 0$;
\item[(iii)] if, in addition, $\opH^{m}(G,H^0(\nu)\otimes H^0(\nu^*)^{(r)}) = 0$
for all $\nu \in X(T)_+$ with $\nu \neq \lambda$, then 
$$\opH^m(\gfpr,k) \cong \opH^m(G,H^0(\lambda)\otimes H^0(\lambda^*)^{(r)}).$$
\end{itemize}
\end{thm}

\subsection{Good Filtrations on Cohomology: } For $\lambda\in X(T)_{+}$ an open problem is to 
determine if $\text{H}^{i}(G,H^{0}(\lambda))^{(-1)}$ admits a good filtration. When $i=0,1,2$ this 
holds for all primes \cite{Jan3,BNP4,W}. For $p>h$, from \cite{AJ} and \cite{KLT}, we have
\begin{equation}\label{ind}
\opH^i(G_1,H^0(\nu))^{(-1)} = 
\begin{cases}
\operatorname{ind}_B^G(S^{\frac{i - \ell(w)}{2}}({\mathfrak u}^*)\otimes\mu) &\text{ if } \nu = w\cdot 0 + p\mu\\
0 & \text{ otherwise,}
\end{cases}
\end{equation}
where ${\mathfrak u} = \operatorname{Lie}(U)$.  Note also that since $p > h$ and $\nu$ is dominant,
$\mu$ must also be dominant. This explicit description of the cohomology shows that in this case the 
cohomology groups $\text{H}^{i}(G,H^{0}(\lambda))^{(-1)}$ admit a good filtration. 

For $n > 0$, let $P_n(\nu)$ denote the number of times that $\nu$ can be expressed as a sum of exactly $n$
positive roots. Note that $P_0(0) = 1$. The function $P_{n}$ is often referred to as 
{\em Kostant's Partition Function}. By using work of \cite{AJ} and the fact that the cohomology 
has a good filtration one can show the following result. 

\begin{prop} Assume $p >h$. Let $\lambda = p \mu + w\cdot 0 \in X(T)_+$. 
Then 
$$
\dim \opH^i(G,H^0(\lambda)\otimes H^0(\lambda^*)^{(1)}) = \sum_{u \in W} (-1)^{\ell(u)} 
P_{\frac{i-\ell(w)}{2}}( u\cdot\lambda - \mu).
$$
\end{prop}

\subsection{Upper Bound on Cohomology: } Now we can use the existence of the filtration 
on ${\mathcal G}(k)=k[G/G({\mathbb F}_{q})]$ and the preceding proposition to provide an 
upper bound on the cohomology group $\opH^i(\gfp,k)$. 

\begin{thm} Assume $p>h$. 
$$\dim \opH^i(\gfp,k) \leq \sum_{\{w \in W | \ell(w) \equiv i \mod 2\}} \sum_{\mu \in X(T)_+} \sum_{u \in W} 
(-1)^{\ell(u)} P_{\frac{i-\ell(w)}{2}}( u\cdot(p\mu+w\cdot 0) - \mu).
$$
\end{thm}

\subsection{} We would like to indicate how to address the question of vanishing ranges (6.1.1) by 
using root combinatorics and Proposition~\ref{vancoho}. 

\begin{prop} Suppose that $\Phi\neq G_2$. Let $p>h$, $\lambda=w\cdot0+p\mu$, $\mu\neq 0$ and 
$\operatorname{H}^i(G,H^0(\lambda)\otimes H^0(\lambda^*)^{(1)})\neq 0$,
\begin{itemize} 
\item[(a)] For $\sigma\in\Phi^+$,
$(p-1)\langle \mu,\sigma^{\vee} \rangle+\ell(w)+\langle w\cdot0,\sigma^{\vee}\rangle \leq i$. 
\item[(b)] For $\widetilde{\alpha}$ the highest~root, $(p-1)\langle \mu,\widetilde{\alpha}^{\vee} \rangle-1\leq i$.
\end{itemize} 
\end{prop}

\begin{proof} (a): First we have  
\begin{eqnarray*}
0\neq \text{H}^i(G,H^0(\lambda)\otimes~H^0(\lambda^*)^{(1)})&\cong & 
\text{Hom}_G(V(\lambda)^{(1)},\text{H}^i(G_{1},H^0(\lambda)))\\
&=& \text{Hom}_G(V(\lambda),\text{ind}_B^G S^{\frac{i-\ell(w)}{2}}({\mathfrak u}^*)\otimes\mu)\\
&=& \text{Hom}_B(V(\lambda),S^{\frac{i-\ell(w)}{2}}({\mathfrak u}^*)\otimes\mu). 
\end{eqnarray*}
Observe that $\lambda-\mu$ must be a weight of $S^{\frac{i-\ell(w)}{2}}({\mathfrak u}^*)$. 
Consequently, $\lambda-\mu$ is expressable as $\frac{i-\ell(w)}{2}$ positive roots. 

Next note that for $\sigma_1,~\sigma_2\in \Phi$,
$\langle \sigma_1,\sigma_2^{\vee} \rangle \leq 2$ $(\Phi\neq G_2)$. 
Therefore, 
\begin{equation*} 
\langle \lambda-\mu,\sigma^{\vee} \rangle \leq(\frac{i-\ell(w)}{2})\cdot2=i-\ell(w).
\end{equation*} 
By substituting $\lambda=p\mu+w\cdot0$ one has 
\begin{equation*}  
\langle p\mu+w\cdot0-\mu,\sigma^{\vee} \rangle \leq i-\ell(w)
\end{equation*}
and 
\begin{equation*} 
(p-1)\langle \mu,\sigma^{\vee} \rangle +\ell(w)+\langle w\cdot 0,\sigma^{\vee} \rangle \leq i. 
\end{equation*} 

(b) The weight $-w\cdot0$ can be expressed uniquely as $l(w)$ distinct positive
roots. Since at most one can be $\widetilde{\alpha}$ and
$\langle \widetilde{\alpha},\widetilde{\alpha}^{\vee} \rangle=2$, it follows that 
\begin{equation*} 
\langle -w\cdot 0,\widetilde{\alpha}^{\vee} \rangle \leq(\ell(w)-1)+2=\ell(w)+1. 
\end{equation*} 
Thus, $\langle  w\cdot 0,\widetilde{\alpha}^{\vee} \rangle \geq -l(w)-1$. 
Now apply part (a) to obtain the result. 
\end{proof} 

\subsection{Applications: } By applying Proposition~\ref{vancoho} and Theorem~\ref{firstclass} with the 
LHS spectral sequence we are able to deduce the following theorem which addresses (6.1.1) and (6.1.2) 
when $\Phi=C_{n}$ and $A_{n}$.  

\begin{thm}[A] Suppose $\Phi$ is of type $C_n$ with $p > 2n$. Then
\begin{itemize}
\item[(a)] $\opH^i(\gfpr,k) = 0$ for $0 < i < r(p - 2)$;
\item[(b)] $\opH^{p-2}(\gfp,k) \cong k$;
\item[(c)] $\opH^{r(p-2)}(\gfpr,k) \neq 0$.
\end{itemize}
\end{thm}

\begin{thm}[B] 
Suppose $\Phi$ is of type $A_n$ with $n \geq 2$.  Suppose further
that $p > n + 1$.
\begin{itemize}
\item[(a)] (Generic case) If $p > n+2$ and $n > 3$, then
\begin{itemize}
\item[(i)] $\opH^i(\gfp, k) = 0$ for $0 < i < 2p-3$;
\item[(ii)] $\opH^{2p-3}(\gfp, k) = k$.
\end{itemize}
\item[(b)]  If $p = n+2$, then
\begin{itemize}
\item[(i)] $\opH^i(\gfp, k) = 0$ for $0 < i < p-2$;
\item[(ii)] $\opH^{p-2}(\gfp, k) = k\oplus k.$ 
\end{itemize}
\item[(c)]  If $n =2$ and $3$ divides $p-1$, then
\begin{itemize}
\item[(i)] $\opH^i(\gfp, k) = 0$ for $0 < i < 2p-6$;
\item[(ii)] $\opH^{2p-6}(\gfp, k) = k \oplus k.$ 
\end{itemize}
\item[(d)]  If $n =2$ and $3$ does not divide $p-1$, then
\begin{itemize}
\item[(i)] $\opH^i(\gfp, k) = 0$ for $0 < i < 2p-3$;
\item[(ii)] $\opH^{2p-3}(\gfp, k) = k.$ 
\end{itemize}
\item[(e)]  If $n =3$ and $p > 5$, then
\begin{itemize}
\item[(i)] $\opH^i(\gfp, k) = 0$ for $0 < i < 2p-6$;
\item[(ii)] $\opH^{2p-6}(\gfp, k) = k.$ 
\end{itemize}
\end{itemize}
\end{thm}

For the type $C_n$ case when $r=1$, we have seen that for $p>h=2n$, 
$\opH^i(\gfpr,k) = 0$ for $0 < i < (p - 2)$. We will now outline how to 
to construct $\lambda$ such that $\text{H}^{p-2}(G,H^0(\lambda)\otimes~H^0(\lambda^*)^{(1)})
\neq 0$. Let $w=s_1s_2\cdots s_{n-1}s_ns_{n-1}\cdots s_1\in W$. Then 
$$-w\cdot0=n\tilde{\alpha}=2n\omega_1=\text{sum~of~all~positive~roots~containing}~\alpha_1.$$
Furthermore, $\lambda=p\omega_1+w\cdot0=(p-2n)\omega_1$, thus 
$\lambda-\omega_1=(\frac{p-1}{2}-n)\tilde{\alpha}$. 

Next one shows that 
\begin{eqnarray*}
P_{\frac{p-1}{2}-n}(u\cdot\lambda-\mu)=\begin{cases} 
1 & \mu=1\\
0 & \mu\neq 1. 
\end{cases}
\end{eqnarray*} 
This implies $\text{H}^{p-2}(G,H^0(\lambda)\otimes H^0(\lambda^*))\neq 0$. Finally, one has to 
also prove that $\lambda\in X(T)_+$ (as above) is the only such weight in 
$X(T)_+$ such that $\text{H}^{p-2}(G,H^0(\lambda)\otimes H^0(\lambda^*)^{(1)})\neq 0$ (see 
\cite[Section 5.3]{BNP8}). 

We remark that (6.1.1) and (6.1.2) for other 
types (when $G$ is simply connected) is addressed further in 
\cite{BNP8}. Moreover, we exhibit uniform bounds when $G$ is of adjoint type.


\begin{thebibliography}{99999999}
\bibliographystyle{amsalpha}

\bibitem[\sf AM]{AM} A. Adem and R.J. Milgram, {\em Cohomology of Finite Groups, 2nd ed.},
Springer Grundlehren der Mathematischen Wissenschaften v. 309,
2004.

\bibitem[\sf Al]{Al} J.L. Alperin, Periodicity in groups, 
{\em Illinois J. Math.}, {\sf 21}, (1977), 776-783. 

\bibitem[\sf And1]{And1} H.H. Andersen, Extensions of modules for algebraic groups, 
{\em American J. Math.}, {\sf 106}, (1984), 489--504. 

\bibitem[\sf And2]{And2} H.H. Andersen, Extensions of simple modules for finite Chevalley groups, 
{\em J. Algebra}, {\sf 111}, (1987), 388--403. 

\bibitem[\sf AJ]{AJ} H.H. Andersen, J.C. Jantzen, 
Cohomology of induced representations for algebraic groups, 
{\em Math. Ann.}, {\sf 269}, (1984), 487--525. 

\bibitem[\sf AS]{AS} G.S. Arvunin, L.L. Scott, Qullen stratification for modules, 
{\em Invent. Math.}, {\sf 66}, 277-286. 

\bibitem[\sf BaKN]{BaKN} I. Bagci, J. Kujawa, D.K. Nakano, Cohomology and 
support varieties for the Lie superalgebra $W(n)$, {\em International Math. 
Research Notices}, doi:10.1093/imrn/rrn115, (2008). 

\bibitem[\sf Ben]{Ben} D.J. Benson, {\em Representations and cohomology}, 
Volumes I and II, Cambridge University Press, 1991.  

\bibitem[\sf BNP1]{BNP1} C.P. Bendel, D.K. Nakano, C. Pillen, On comparing the cohomology 
of algebraic groups, finite Chevalley groups and Frobenius kernels, 
{\em J. Pure and Applied Algebra}, {\sf 163}, (2001), 119-146. 

\bibitem[\sf BNP2]{BNP2} C.P. Bendel, D.K. Nakano, C. Pillen, Extensions 
for finite Chevalley groups II, {\em Transactions of the AMS}, {\sf 354}, (2002), 4421-4454. 

\bibitem[\sf BNP3]{BNP3} C.P. Bendel, D.K. Nakano, C. Pillen, Extensions for 
finite Chevalley groups I, {\em Advances in Math.}, {\sf 183}, (2004), 380--408.  

\bibitem[\sf BNP4]{BNP4} C.P. Bendel, D.K. Nakano, C. Pillen, Extensions for 
Frobenius kernels, {\em J. Algebra}, {\sf 272}, (2004), 476--511. 

\bibitem[\sf BNP5]{BNP5} C.P. Bendel, D.K. Nakano, C. Pillen, Extensions for 
finite groups of Lie type: Twisted Groups, {\em Finite Groups 2003} (held at the 
University of Florida), de Gruyter, New York, 2004, 29--46.  

\bibitem[\sf BNP6]{BNP6} C.P. Bendel, D.K. Nakano, C. Pillen, Extensions for 
finite groups of Lie type II: Filtering the truncated induction functor, 
{\em Cont. Math.}, {\sf 413}, (2006), 1--23. 

\bibitem[\sf BNP7]{BNP7} C.P. Bendel, D.K. Nakano, C. Pillen, Second cohomology 
for Frobenius kernels and related structures, {\em Advances in Math.}, {\sf 209}, 
(2007), 162--197 

\bibitem[\sf BNP8]{BNP8} C.P. Bendel, D.K. Nakano, C. Pillen, On the vanishing ranges for 
the cohomology of finite groups of Lie type, preprint. 

\bibitem[\sf BNP9]{BNP9} C.P. Bendel, D.K. Nakano, C. Pillen, On the vanishing ranges 
for the cohomology of finite groups of Lie type II, in preparation. 

\bibitem[\sf BNPP]{BNPP} C.P. Bendel, D.K. Nakano, B.J. Parshall, C. Pillen, Quantum 
group cohomology via the geometry of the nullcone, preprint. 

\bibitem[\sf BE]{BE} P.A. Bergh, K. Erdmann, The Avrunin-Scott theorem for quantum complete 
intersections, {\em J. Algebra}, {\sf 322}, (2009), 479-488.  

\bibitem[\sf BKN1]{BKN1} B.D. Boe, J. Kujawa, D.K. Nakano, Cohomology and 
support varieties for Lie superalgebras, {\em to appear in Transactions of the AMS}.  

\bibitem[\sf BKN2]{BKN2} B.D. Boe, J. Kujawa, D.K. Nakano, Cohomology and 
support varieties for Lie superalgebras II, {\em Proceedings of the London Math. Soc.}, {\sf 98}, (2009), 19--44.  

\bibitem[\sf BKN3]{BKN3} B.D. Boe, J. Kujawa, D.K. Nakano, Complexity 
and module varieties for classical Lie superalgebras, preprint.  


\bibitem[\sf Ca1]{Ca1} J.F. Carlson, The varieties and cohomology ring of 
a module, {\em J. Algebra}, {\sf 85}, (1983), 104-143. 

\bibitem[\sf Ca2]{Ca2} J.F. Carlson, The variety of an indecomposable module 
is connected, {\em Invent. Math.}, {\sf 77}, (1984), 291-299. 

\bibitem[\sf Ca3]{Ca3} J.F. Carlson, {\em Module varieties and cohomology 
rings of finite groups}, Vorlesungen aus dem Fachbereich Mathematick der 
Universitat Essen, 1985. 

\bibitem[\sf CLN]{CLN} J.F. Carlson, Z. Lin, D.K. Nakano,
Support varieties for modules over Chevalley groups and classical 
Lie algebras, {\em Transactions of the AMS}, {\sf 360}, (2008), 1879--1906. 

\bibitem[\sf CLNP]{CLNP} J.F. Carlson, Z. Lin, D.K. Nakano, B.J. Parshall, 
The restricted nullcone, {\em Cont. Math.}, {\sf 325}, (2003), 51-73. 

\bibitem[\sf CMN]{CMN} J.F. Carlson, N. Mazza, D.K. Nakano, Endotrivial 
groups for finite Chevalley groups, {\em J. Reine Angew. Math.}, {\sf 595} 
(2006), 93--120.  

\bibitem[\sf Car]{Car} R.W. Carter, {\em Finite groups of Lie
type}, Wiley-Interscience,  1985.

\bibitem[\sf Ch]{Ch} L. Chastkofsky, Characters of projective
indecomposable modules for finite Chevalley groups, {\em Proc.
Sympos. Pure Math.}, {\sf 37}, (1980), 359-362.

\bibitem
[\sf CM]{CM} D.H. Collingwood, W.M. McGovern,
{\em Nilpotent Orbits in Semisimple Lie Algebras}, Van Nostrand  
Reinhold,
1993.

\bibitem[\sf CPS]{CPS} E.T. Cline, B.J. Parshall, L.L. Scott, Finite dimensional algebras 
and highest weight categories, {\em J. Reine Angew. Math.}, {\sf 391}, (1988), 85-99. 

\bibitem[\sf CPSvdK]{CPSvdK} E. Cline, B. Parshall, L. Scott, W. van der Kallen,
Rational and generic cohomology, {\em Invent. Math.},
{\sf 39}, (1977), 143-163.

\bibitem[\sf EH]{EH} K. Erdmann, M. Holloway, Rank varieties and projectivity 
for a class of local algebras, {\em Math. Zeitschrift}, {\sf 247}, (2004), 
67--87.  

\bibitem[\sf F1]{F1} E.M. Friedlander, Computations of $K$-theories of finite 
fields, {\em Topology}, {\sf 15}, (1976), 87-109.

\bibitem[\sf F2]{F2} E.M. Friedlander, Weil restriction and support varieties, preprint 

\bibitem[\sf FP1]{FP1} E.M. Friedlander, B.J. Parshall, Support varieties 
for restricted Lie algebras, {\em Invent. Math.}, {\sf 86}, (1986), 553-562. 

\bibitem[\sf FP2]{FP2} E.M. Friedlander, B.J. Parshall, 
Cohomology of Lie algebras and algebraic groups, 
{\em American J. Math.}, {\sf 108}, (1986), 235-253. 

\bibitem[\sf FP3]{FP3} E.M. Friedlander, B.J. Parshall, Cohomology of
infinitesimal and discrete groups, {\em Math. Ann.}, {\sf 273}, (1986), 
353--374.

\bibitem
[\sf FPe]{FPe} E.M. Friedlander, J. Pevtsova, 
Representation theoretic support spaces for finite group 
schemes, {\em American J. Math.}, {\sf 127}, (2005), 379--420. 


\bibitem
[\sf FS]{FS} E.M. Friedlander, A. Suslin, Cohomology of finite 
group schemes over a field, {\em Invent. Math.}, {\sf 127}, 
(1997), no. 2, 209-270.  

\bibitem[\sf GH]{GH} R.M. Guralnick, C. Hoffman, The first cohomology group and generation of simple groups, 
{\em Groups and geometries} (Siena, 1996),  81-89, Trends Math., Birkhauser, Basel, 1998

\bibitem[\sf GT]{GT} R.M. Guralnick, P.H. Tiep, First cohomology groups of Chevalley groups in cross characteristic, 
preprint. 

\bibitem[\sf H]{H} H.L. Hiller, Cohomology of Chevalley groups over finite fields,
{\em J. Pure Appl. Alg.}, {\sf 16}, (1980), 259-263.

\bibitem
[\sf Hu1]{Hu1} J.E. Humphreys, {\em Conjugacy Classes in Semisimple Algebraic Groups}, 
Math. Surveys Monographs, vol. 43, Amer. Math. Soc., Providence, RI, 1995. 

\bibitem
[\sf Hu2]{Hu2} J.E. Humphreys, {\em Modular Representations of Finite Groups of Lie Type}, 
London Math. Soc. Lecture Note Series, {\sf 326}, Cambridge University Press, 2005. 

\bibitem[\sf Jan1]{Jan1} J.C. Jantzen, {\em Representations of Algebraic 
Groups}, American Mathematical Society, Mathematical Surveys and Monographs, 
Vol. 203, 2003.

\bibitem[\sf Jan2]{Jan2} J.C. Jantzen, Support varieties of Weyl modules, 
{\em Bull of London Math. Soc.}, {\sf 19}, (1987), 238-244. 

\bibitem[\sf Jan3]{Jan3} J.C. Jantzen, First cohomology groups for
classical Lie algebras, {\em Progress in Mathematics}, {\sf 95},
Birkhauser, 1991, 289-315.

\bibitem[\sf Jan4]{Jan4} J.C. Jantzen, Low dimensional representations of
reductive groups are semisimple, {\em Algebraic Groups and Lie Groups}, ed.
G.I. Lehrer, Cambridge University Press, 1997, 255-266.

\bibitem[\sf Jan5]{Jan5} J.C. Jantzen, Nilpotent orbits in representation theory, {\em Lie theory},  
Progr. Math., {\sf 228}, 1-211, Birkhäuser Boston, Boston, MA, 2004


\bibitem[\sf KLT]{KLT} S. Kumar, N. Lauritzen, J. Thomsen, Frobenius  
splitting of cotangent bundles of flag varieties, {\em Invent. Math.}, 
{\sf 136}, (1999), 603-621.

\bibitem[\sf LN1]{LN1} Z. Lin, D.K. Nakano, Complexity for modules over 
finite Chevalley groups and classical Lie algebras, 
{\em Invent. Math.}, {\sf 138}, (1999), 85-101. 

\bibitem[\sf LN2]{LN2} Z. Lin, D.K. Nakano, Projective modules for Chevalley groups 
and Frobenius kernels, {\em Bulletin of the London Math. Soc.}, {\sf 39}, (2007), 1019--1028. 

\bibitem[\sf McG]{McG} W. McGovern, Rings of regular functions on nilpotent 
orbits and their covers, {\em Invent. Math.}, {\sf 97}, (1989), 209--217. 

\bibitem[\sf NPal]{NPal} D.K. Nakano, J.H. Palmieri, 
Support varieties for the Steenrod algebra, {\em Math. Zeit.}, {\sf 227} 
(1998), 663-684. 

\bibitem[\sf NPV]{NPV} D.K. Nakano, B.J. Parshall, D.C. Vella, 
Support varieties for algebraic groups, {\em J. Reine Angew. Math.}, 
{\sf 547}, (2002), 15-49. 

\bibitem[\sf Ost]{Ost} V. Ostrik, Support varieties for quantum groups, 
{\em Funct. Anal. Appl.}, {\sf 32}, (1998), 237--246. 

\bibitem[\sf Par]{Par} B.J. Parshall, 
Cohomology of algebraic groups, 
{\em Proc. Symp. Pure Math.}, {\sf 47}, (1987), 233-248. 

\bibitem[\sf Q1]{Q1} D.G. Quillen, The spectrum of an equivariant cohomology ring I, 
{\em Ann. Math.}, {\sf 94}, (1971), 549-572.

\bibitem[\sf Q2]{Q2} D.G. Quillen, The spectrum of an equivariant cohomology ring II, 
{\em Ann. Math.}, {\sf 94}, (1971), 573-602.

\bibitem[\sf Q3]{Q3} D.G. Quillen, On the cohomology and $K$-theory of finite fields,
{\em Ann. Math.}, {\sf 96}, (1972), 551-586.

\bibitem[\sf Sc]{Sc} L.L. Scott, Some new examples in 1-cohomology, {\em J. Algebra}, {\sf 260}, 
(2003), 416-425. 

\bibitem[\sf Sei]{Sei} G. Seitz, Unipotent elements, tilting modules, and saturation, 
{\em Invent. Math}, {\sf 141}, (2000), 467-502. 

\bibitem
[\sf Sin1]{Sin1} P. Sin, Extensions of simple modules for ${\rm SP}_{4}(2^{n})$ and 
${\rm Suz}_{3}(2^{m})$, {\em Bulletin of the London Math. Soc.}, {\sf 24}, (1992), 159--164.  

\bibitem
[\sf Sin2]{Sin2} P. Sin, Extensions of simple modules for ${\rm SL}_{3}(2^{n})$ and 
${\rm SU}_{3}(2^{n})$, {\em Proceedings of the London Math. Soc.}, {\sf 65}, (1992), 265--296.  

\bibitem
[\sf Sin3]{Sin3} P. Sin, Extensions of simple modules for ${G}_{2}(3^{n})$ and 
${}^{2}{G}_{2}(3^{m})$, {\em Proceedings of the London Math. Soc.}, {\sf 66}, (1993), 327--357.  

\bibitem
[\sf Sin4]{Sin4} P. Sin, On the $1$-cohomology of the groups $G_{2}(2^{n})$, 
{\em Comm. Algebra}, {\sf 20}, (1992), 2653--2662. 

\bibitem
[\sf Sin5]{Sin5} P. Sin, The cohomology in degree $1$ of the group $F_{4}$ in 
characteristic $2$ with coefficients in a simple modules, 
{\em J. Algebra}, {\sf 164}, (1994), 695--717. 

\bibitem
[\sf Sm]{Sm} S.D. Smith, Sheaf homology and complete reducibility, {\em J. Algebra}, 
{\sf 95}, (1985), 72--80. 

\bibitem
[\sf SS]{SS} N. Snashall, O Solberg, Support varieties and Hochschild 
cohomology rings, {\em Proceedings of the London Math. Soc.}, {\sf 88}, (2004), 705--732.  


\bibitem[\sf SFB1]{SFB1} A. Suslin, E.M. Friedlander, C.P. Bendel, 
Infinitesimal one-parameter subgroups and cohomology, {\em Journal of the 
AMS}, {\sf 10}, (1997), no. 3, 693--728.

\bibitem[\sf SFB2]{SFB2} A. Suslin, E.M. Friedlander, C.P. Bendel, 
Support varieties for infinitesimal group schemes, {\em Journal of the AMS}, 
{\sf 10}, (1997), no. 3, 729--759.

\bibitem
[\sf UGA1]{UGA1} University of Georgia VIGRE Algebra Group,
Varieties of nilpotent elements for simple Lie algebras I:
Good Primes, {\em J.\ Algebra}, {\sf 280}, (2004), 719--737.

\bibitem
[\sf UGA2]{UGA2} University of Georgia VIGRE Algebra Group,
Varieties of nilpotent elements for simple Lie algebras II:
Bad Primes, {\em J.\ Algebra}, {\sf 292}, (2005), 65--99.

\bibitem
[\sf UGA3]{UGA3} University of Georgia VIGRE Algebra Group,
Support varieties for Weyl modules over bad primes, {\em J. Algebra}, 
{\sf 312}, (2007), 602--633.

\bibitem
[\sf W]{W} C.B. Wright, Second cohomology groups for Frobenius kernels, 
Ph.D Thesis, University of Georgia, (2008). 



\end{thebibliography}
\end{document}